\documentclass[12pt]{amsart}
\usepackage{amssymb}
\usepackage{amsmath,amscd}

\let\frak\mathfrak
\let\Bbb\mathbb

\def\>{\relax\ifmmode\mskip.666667\thinmuskip\relax\else\kern.111111em\fi}
\def\<{\relax\ifmmode\mskip-.333333\thinmuskip\relax\else\kern-.0555556em\fi}
\def\vsk#1>{\vskip#1\baselineskip}
\def\vv#1>{\vadjust{\vsk#1>}\ignorespaces}
\def\vvn#1>{\vadjust{\nobreak\vsk#1>\nobreak}\ignorespaces}

\let\Medskip\medskip
\def\medskip{\par\Medskip}
\let\Bigskip\bigskip
\def\bigskip{\par\Bigskip}

\let\Maketitle\maketitle
\def\maketitle{\hrule height0pt\vskip-\baselineskip
\Maketitle\thispagestyle{empty}\let\maketitle\empty}

\newtheorem{thm}{Theorem}[section]

\newtheorem{lem}[thm]{Lemma}

\numberwithin{equation}{section}

\theoremstyle{definition}

\let\mc\mathcal
\let\nc\newcommand

\nc{\on}{\operatorname}
\nc{\Z}{{\mathbb Z}}
\nc{\C}{{\mathbb C}}
\nc{\N}{{\mathbb N}}
\nc{\pone}{{\mathbb C}{\mathbb P}^1}
\nc{\arr}{\rightarrow}
\nc{\larr}{\longrightarrow}
\nc{\al}{\alpha}
\nc{\W}{{\mc W}}
\nc{\la}{\lambda}
\nc{\su}{\widehat{{\mathfrak sl}}_2}
\nc{\g}{{\mathfrak g}}
\nc{\h}{{\mathfrak h}}
\nc{\m}{{\mathfrak m}}
\nc{\n}{{\mathfrak n}}
\nc{\Gm}{\Gamma}
\nc{\La}{\Lambda}
\nc{\gl}{\widehat{\mathfrak{gl}_2}}
\nc{\bi}{\bibitem}
\nc{\om}{\omega}
\nc{\Res}{\on{Res}}
\nc{\gm}{\gamma}
\nc{\Om}{\Omega}

\def\Res{\on{Res}}

\def\B{{\mc B}}

\def\E{{\mc E}}

\def\V{{\mc V}}

\let\Dl\Delta

\let\leq\leqslant

\nc{\gln}{\mathfrak{gl}_N}
\nc{\sln}{\mathfrak{sl}_N}

\def\beq{\begin{equation}}
\def\eeq{\end{equation}}
\def\be{\begin{equation*}}
\def\ee{\end{equation*}}

\nc{\bean}{\begin{eqnarray}}
\nc{\eean}{\end{eqnarray}}
\nc{\bea}{\begin{eqnarray*}}
\nc{\eea}{\end{eqnarray*}}
\nc{\bs}{\boldsymbol}
\nc{\Ref}[1]{{\rm(\ref{#1})}}
\nc{\glN}{\mathfrak{gl}_N}
\nc{\glNt}{\mathfrak{gl}_N[t]}
\nc{\s}{sing}
\nc{\R}{\Bbb R}

\nc{\Oml}{{\Om_{\bs\la}}}
\nc{\OmLb}{{\Om_{\bs\La,\bs\la,\bs b}}}
\nc{\Ol}{{\mc O_{\bs\la}}}
\nc{\OLb}{{\mc O_{\bs\La,\bs\la,\bs b}}}
\nc{\VSl}{{(\V^S)_{\bs\la}}}
\nc{\Bl}{{\B_{\bs\la}}}
\nc{\Ml}{{\mc M_{\bs\la}}}
\nc{\Mlb}{{\mc M_{\bs\La,\bs\la,\bs b}}}
\nc{\Blb}{{\B_{\bs\La,\bs\la,\bs b}}}
\nc{\Omn}{{\Omega_{\bs n,\bs b,\bs K}}}
\nc{\Omlb}{{\bar\Om_{\bs\la}}}
\nc{\ep}{\epsilon}

\nc{\Dlb}{\Dl_{\bs\La,\bs\la,\bs b,\bs K}}
\nc{\Bb}{{\bf b}}
\nc{\glt}{{\frak{gl}_2}}
\nc{\A}{{\mc A}}
\nc{\slt}{{\frak{sl}_2}}
\nc{\Ma}{{\mc M_{\bs a}}}
\nc{\Mal}{{\mc M_{\bs\la,\bs a}}}
\nc{\Malp}{{\mc M_{\phi,\bs\la,\bs a}}}

\nc{\Bal}{{\B_{\bs\la,\bs a}}}
\nc{\Ola}{{\mc O_{\bs\la,\bs a}}}
\nc{\Bv}{{\mc B_{\V^S}}}
\nc{\Bvz}{{\mc B^0_{\V^S}}}

\nc{\sing}{{\rm Sing\,}}
\nc{\Uglt}{U(\glt)}
\nc{\Olo}{{\mc O^0_{\bs\la}}}
\nc{\kk}{K}

\nc{\Oll}{{\Omega_{\bs\la}}}
\nc{\T}{{\mc T}}

\nc{\CC}{{\mc C}}
\nc\Vl{{(\V^S)^{sing}_{\bs\la}}}

\nc{\PP}{{\Bbb P}}
\nc{\LL}{{\mc L}}
\nc{\FF}{{\mc F}}

\nc{\zz}{{\bs z}}
\nc{\TT}{{\bs t}}
\DeclareMathOperator{\res}{Res}
\nc{\CCu}{{\CC_{\tilde{\bs u}}}}
\nc{\tu}{{\tilde{\bs u}}}
\nc{\OS}{{\mc A}}
\nc{\LLw}{{\LL_{\bs w}}}

\usepackage[all]{xy}
\xyoption{web}

\begin{document}

\title[Cohomology of the complement to an elliptic  arrangement]
{Cohomology of the complement to an elliptic  arrangement}

\author[A.\,Levin, A.\,Varchenko]
{A.\,Levin$^\star$, A.\,Varchenko$^\diamond$}

\maketitle

\begin{center}
{\it  $^\star$\,
State University - Higher School of Economics, Department of Mathematics,\\
20 Myasnitskaya Street, Moscow, 101000, Russia\/}

\medskip
{\it  $^\star$\,
Laboratory of Algebraic Geometry, GU-HSE,
7 Vavilova Street, Moscow, 117312, Russia\/}

\medskip

{\it $^\diamond$\,
Department of Mathematics, University of North Carolina
at Chapel Hill\\ Chapel Hill, NC 27599-3250, USA\/}
\end{center}

{\let\thefootnote\relax
\footnotetext{\vsk-.8>\noindent
 $^\star$\,Supported in part by AG Laboratory GU-HSE, RF government
grant, ag. 11 11.G34.31.0023
\\
$^\diamond$\,Supported in part by NSF grant DMS-1101508

}

\medskip

\begin{abstract}
We consider the complement to an arrangement of hyperplanes in a cartesian power of an elliptic curve
and describe its cohomology with coefficients in a nontrivial rank one local system.

\end{abstract}

\maketitle

\bigskip

\section{Introduction}

We start with  a cartesian power $E^k$ of an elliptic curve $E$ and a nontrivial rank one local system on $E^k$.
We consider an arrangement of elliptic hyperplanes in $E^k$ and describe the cohomology of its complement with coefficients in
the local system.We show that the cohomology is nontrivial only in degree $k$. We present each cohomology class by a unique closed
holomorphic differential form. Our forms are elliptic analogs of the Arnold-Brieskorn-Orlik-Solomon logarithmic differential forms
representing cohomology classes of the complement to an arrangement of hyperplanes in an affine space. For the elliptic discriminantal arrangement
our forms are the forms considered in \cite{FV1, FV2, FRV} to solve the KZB equations in hypergeometric integrals and to construct Bethe eigenfunctions
to the elliptic Calogero-Moser operators.

To simplify the exposition, we first consider in
 Sections \ref{sec 2} and \ref{sec 3} the case of an elliptic discriminantal arrangement, then in Sections \ref{sec 4} and \ref{proofs} we consider
 arbitrary elliptic arrangements.

\medskip

The authors thank the Max Planck Institute for Mathematics in Bonn for hospitality.

\section{Cohomology of an elliptic discriminantal arrangement}
\label{sec 2}

Fix a natural number $k$ and $\tau\in\C$, Im $\tau>0$.
Denote $\Lambda = \tau\Z+\Z\subset \C$.
The group $\Gamma=\Z\oplus \Z$ acts on $\C$ by transformations
$(l,m): t\mapsto t+l\tau+m$. The action on each factor gives an action
of $\Gamma^k$ on $\C^k$. Denote by $p: \C^k\to\C^k/\Gamma^k$
the canonical projection onto the space of orbits.
We have $\C^k/\Gamma^k=E^k$, where $E$ is the elliptic curve $\C/\Gamma$.

For each representation
$\rho$
 of $\Gamma^k$ on a vector space $W$ we get a vector bundle over $E^k$
with a flat connection, which is $(\C^k\times W)/\Gamma^k \to
\C^k/\Gamma^k$.
In particular, we may fix complex numbers $\bs w=(w_1,\dots,w_k)$, take $W=\C$, and
$\rho_{\bs w}(\gamma)=e^{2\pi \sqrt{-1} (w_1l_1+\dots+w_kl_k)}$ for $\gamma=(l_1,m_1)\times\dots\times(l_k,m_k)$.
This line bundle over $E^k$ with the flat connection will be denoted by $\LLw$.

We say that the numbers $\bs w=(w_1,\dots,w_k)$ are {\em discriminantal convenient} if for any subset $I\subset\{1,\dots,k\}$ the sum
$\sum_{i\in I}w_i$ is not in $\Lambda$.

\medskip

Fix distinct complex numbers $\zz=(z_1,\dots,z_n)$. The {\em discriminantal} arrangement $\CC_\zz$ in $\C^k$ with
parameters $\zz$ is the arrangement of hyperplanes:
\bea
&&
H^a_i\ :\
t_i-z_a=0,
\qquad i=1,\dots,k,\ a=1,\dots,n;
\\
&&
H_{ij}\ :\
t_i-t_j=0,
\qquad 1\leq i<j\leq k.
\eea
Let $M_\zz$ denote its  complement  $\C^k-\cup_{H\in \CC_\zz}H$.

The $\Gamma^k$-orbit of $\CC_\zz$ is the infinite arrangement $\CC_{\zz,\Gamma^k}=\{\gamma(H)\ |\
\gamma\in\Gamma^k, H\in\CC\}$. Denote by $M_{\zz,\Gamma^k}$ its complement
$\C^k-\cup_{\gamma\in\Gamma^k,H\in\CC_\zz} \gamma(H)$. Denote by $\tilde M_{\zz,\tau}\subset E^k$
 the image of $M_{\zz,\Gamma^k}$ under the projection $p$.

\begin{thm}
\label{thm main}
Assume that the numbers $\bs w$ are
discriminantal convenient and $z_1,\dots,z_n$ project to distinct points of $E$. Then
 $H^\ell(\tilde M_{\zz,\tau};\LLw)=0$ for $\ell\neq k$ and
$H^k(\tilde M_{\zz,\tau};\LLw)$ is canonically isomorphic to $H^k(M_\zz;\C)$,
the $k$-th cohomology group
with trivial coefficients of the complement in $\C^k$ to the discriminantal arrangement.
\end{thm}

Here $H^*(\tilde M_{\zz,\tau};\LLw)$ denotes the cohomology of $\tilde M_{\zz,\tau}$ with coefficients in
the local system of horizontal sections of $\LLw$.
Theorem \ref{thm main} is proved in Section \ref{proofs}.

\medskip

The space $\tilde M_{\zz,\tau}$ is a $K(\pi, 1)$-space and our theorem describes the cohomology of the fundamental group
of $\tilde M_{\zz,\tau}$ with coefficients in $\LLw$.
Notice that the fundamental group  of $\tilde M_{\zz,\tau}$ is a subgroup of the pure elliptic braid group.
\medskip

The cohomology  $H^k(M_\zz;\C)$ of the complement to the discriminantas arrangement in $\C^k$  are
presented be explicit  logarithmic forms  by the Arnold-Brieskorn-Orlik-Solomon theory.
Below we describe logarithmic differential forms
representing elements of
$H^k(\tilde M_{\zz,\tau};\LLw)$. Those forms were used in \cite{FV1, FV2, FRV} to give integral hypergeometric
representations for solutions
of the KZB equations with values in  a tensor product of highest weight representations
of a simple Lie algebra and to construct Bethe eigenfunctions of elliptic Calogero-Moser operators.

\section{Differential forms of a discriminantal arrangement}
\label{sec 3}

In this section we follow \cite{FRV}. Theorem \ref{thm on forms} is new.

\subsection{Combinatorial space}
An {\em ordered $k$-forest} is a graph with no cycles, with $k$
edges, and a numbering of its edges by the numbers $1,2,\ldots,k$.
We consider the ordered $k$-forests on the vertex set of symbols
$\{z_1,\dots,z_n,t_1,\ldots,t_k\}$. An ordered  $k$-forest $T$ is {\em admissible} if
all $t_1,\dots,t_k$ are among the vertices of $T$ and each connected component
of $T$ has exactly one vertex
from the subset $\{z_1,\dots,z_n\}$.

Let $\A^k_n$ be the complex vector space generated by the admissible ordered
$k$-forests, modulo the following  relations:
\begin{description}
\item[R1]
$T_1=-T_2$ if $T_1$ and $T_2$ have the same underlying graph, and
the order of their edges differ by a transposition;
\item[R2]
\begin{equation}
\label{triangle}
\begin{picture}(205,20)(0,15)
\put(15,0){$*$} \put(18,3){\line(-1,2){15}} \put(0,30){$*$}
\put(18,3){\line(1,2){15}}\put(30,30){$*$} \put(3,15){$a$}
\put(28,15){$b$}\put(48,15){$+$}

\put(85,0){$*$} \put(88,3){\line(-1,2){15}} \put(70,30){$*$}
\put(73,33){\line(1,0){30}}\put(100,30){$*$} \put(73,15){$b$}
\put(85,34){$a$} \put(118,15){$+$}

\put(155,0){$*$} \put(143,33){\line(1,0){30}} \put(140,30){$*$}
\put(158,3){\line(1,2){15}}\put(170,30){$*$} \put(155,34){$b$}
\put(168,15){$a$} \put(190,15){$=0$}
\end{picture}\qquad\qquad (a,b\in\{1,\ldots,k\}),
\notag
\end{equation}

\vskip  .4 true cm
\noindent
that is, the sum of three $k$-forests
that locally (i.e. their subgraphs spanned by 3 vertices) differ as
above, but are otherwise identical, is 0.
\end{description}

A linear map $\phi$ of $\A^k_n$ to a vector space $W$ is called
a {\em representation} of $\A^k_n$.
Suppose we are given a vector space $W$ and a vector $\phi(T)\in W$
for every admissible $k$-forest $T$.
This data induces a representation  if the
assignment $T\mapsto \phi(T)$ respects relations R1 and R2.

\subsection{Rational representation}

Let $e$ be an edge of an admissible forest $T$. The connected
component of $T$,  containing $e$, has  exactly one vertex, say $z_a$, from the
set $\{z_1,\dots,z_n\}$. Denote by $h(e)$ and $t(e)$ the head and
tail of the edge $e$, i.e. the vertices adjacent to $e$, farther
resp. closer to the vertex $z_a$.

\medskip
Fix distinct complex numbers $\zz=(z_1,\dots,z_n)$.
To an admissible forest $T$ with ordered edges $e_1,\dots,e_k$,
we assign a closed holomorphic differential $k$-form $\phi_{rat}(T)$ on $M_\zz$ by the formula
\bea
\phi_{rat}(T)=\wedge_{i=1}^{k} d \log(h(e_i)-t(e_i)).
\eea
This assignment defines a representation of $\A^k_n$ on the space of  $k$-forms on $M_\zz$, see \cite{A, OS}.  By  \cite{A, OS},
the representation is an isomorphism onto its image, see Proposition 2.1 in \cite{FRV}. We denoted the image by $\A^k_\zz$.
 The assignment
to a form its cohomology class gives a linear map  $\A^k_n\to H^k(M_\zz;\C)$.

\begin{thm} [\cite{A, OS}]
\label{thm OS}
The map   $\A^k_n\to H^k(M_\zz;\C)$ is an isomorphism.
\qed
\end{thm}

\subsection{Theta representation.}
\label{theta}

For $z,\tau\in\C$, Im $\tau>0$, the first Jacobi theta function is
defined by the infinite product
\bea
&&
\phantom{aaa}
\theta(z)=\theta(z,\tau)=\sqrt{-1}e^{\pi\sqrt{-1}
(\tau/4-z)}(x;q)(\frac{q}{x};q)(q;q),
\\
&&
q=e^{2\pi \sqrt{-1}\tau},\
x=e^{2\pi \sqrt{-1} z},
\qquad
\phantom{aa}
(y;q)=\prod_{j=0}^{\infty}(1-yq^j),
\eea
\cite{ww}.
It is an entire holomorphic function of $z$ satisfying
\[
\theta(z+1,\tau)=-\theta(z,\tau), \quad
\theta(z+\tau,\tau)=-e^{-\pi \sqrt{-1} \tau-2\pi \sqrt{-1} z}\theta(z,\tau),\quad
\theta(-z,\tau)=-\theta(z,\tau).
\]
By $\theta'(z,\tau)$ we will mean the derivative in the $z$
variable.
Define
\[
\sigma_w(t)=\sigma_w(t,\tau)=\frac{\theta(w-t,\tau)}{\theta(w,\tau)\theta(t,\tau)}\cdot
\theta'(0,\tau).
\]
The listed properties of the theta function yield that the function $\sigma$
-- viewed as a function of $t$ -- has simple poles at the
points of $\Lambda\subset \C$, as well as the
properties
\begin{equation}
\sigma_w(t+1,\tau)=\sigma_w(t,\tau), \qquad
\sigma_w(t+\tau,\tau)=e^{2\pi \sqrt{-1} w}\sigma(t,\tau),\qquad
\res_{t=0}\sigma_w(t,\tau)=1.
\label{sigprop}
\notag
\end{equation}
We also have
\begin{equation}\label{sigma_identity}
\sigma_{w_1+w_2}(t-u)\sigma_{w_2}(s-t)-\sigma_{w_2}(s-u)\sigma_{w_1}(t-u)+\sigma_{w_1}(t-s)\sigma_{w_1+w_2}(s-u)=0,
\notag
\end{equation}
see for example, \cite{FRV}.

\medskip
Fix discriminantal convenient
 complex numbers $\bs w=(w_1,\dots,w_k)$ and distinct complex numbers
$\zz=(z_1,\dots,z_n)$. For $i=1,\dots,k$, we say that  $t_i$ has {\em weight} $w_i$.

Let $T$ be an admissible forest and $v$ a vertex of $T$.  The connected
component of $T$,  containing $v$, has  exactly one vertex, say $z_a$, from the
set $\{z_1,\dots,z_n\}$.  We
define the {\em branch} $B(v)$ of $v$ to be the collection of those
vertices $u$ for which the unique path connecting $u$ with $z_a$
contains $v$. By definition $v\in B(v)$.  The {\em load} $L(v)$ of a
vertex $v$ in the forest $T$ is defined to be the sum of the weights of
the vertices in $B(v)$.

To an admissible forest $T$ with ordered edges $e_1,\dots,e_k$,
we assign a closed holomorphic differential $k$-form $\phi_{\theta}(T)$ on $M_{\zz,\Gamma^k}$ by the formula
\bea
\phi_{\theta}(T) =\wedge_{i=1}^{k} \sigma_{L(h(e_i))}(h(e_i)-t(e_i),\tau)\ d(h(e_i)-t(e_i)).
\eea
Notice that if $\bs w$ are discriminantal convenient,
 then the load of each vertex $h(e_i)$ does not lie in $\Lambda$ and the form is well-defined.

\begin{thm}
[\cite{FRV}] Assume that $\bs w$ is
discriminantal  convenient and $z_1,\dots,z_n$ project to distinct points of $E$. Then
the assignment $T\mapsto \phi_\theta(T)$
defines a representation of $\A^k_n$ on the space of  $k$-forms on $M_{\zz,\Gamma^k}$.
The representation is an isomorphism onto its image, denoted by $\A^k_{\zz,\Gamma^k}$.
Each element of $\A^k_{\zz,\Gamma^k}$ descends to a closed holomorphic differential form on $\tilde M_{\zz,\tau}$ with values
in $\LLw$.
\qed
  \end{thm}

The assignment
to a form its cohomology class defines a linear map  $\A^k_n\to H^k(\tilde M_{\zz,\tau};\LLw)$.
In Section \ref{proofs} the following theorem will be proved.

\begin{thm}
\label{thm on forms}
Assume that the numbers $\bs w$ are
discriminantal convenient and $z_1,\dots,z_n$ project to distinct points of $E$. Then
the map   $\A^k_n\to H^k(\tilde M_{\zz,\tau};\LLw)$ is an isomorphism.
\end{thm}

Theorems \ref{thm OS} and \ref{thm on forms} imply the second statement of Theorem \ref{thm main}.

\medskip

According to \cite{SV}, $\dim \A^k_n = \sum_{m_1+\dots+m_n=k} m_1!\dots m_n!$\ .

\section{Transversal elliptic hyperplanes}
\label{sec 4}

\subsection{Elliptic hyperplanes in $E^k$} Denote $\E=E^k$. Any  $k\times k$-matrix $C\in GL(k,\Z)$ defines an isomorphism
$\E\to\E,\ {} (\tilde t_1,\dots,\tilde t_k) \mapsto (\sum_j \tilde t_jc_{j1},$ $\dots,\sum_j\tilde t_jc_{jk})$. The collection
$\sum_j\tilde t_jc_{j1},\dots,\sum_j \tilde t_jc_{jk}$ will be called {\em coordinates} on $\E$.

Let $\tilde t_1',\dots,\tilde t_k'$ be coordinates on $\E$. Fibers of the projection $\E\to E^{\ell}$ along the last
$k-\ell$ coordinates will be called {\em  elliptic $k-\ell$-planes} in $\E$, in particular,
elliptic  $k-1$-planes are {\em elliptic hyperplanes}.
Fibers of the same projection will be called {\em parallel} $k-\ell$-planes.

An elliptic $k-\ell$-plane is defined by equations $\tilde t_i'=\tilde z_i,$\ $ i=1,\dots, \ell$, for suitable $\tilde z_i\in E$.
Each elliptic $k-\ell$-plane is isomorphic to $E^{k-\ell}$ as an algebraic variety.

\begin{lem}
\label{lem normal}
The normal bundle of an elliptic $k-\ell$-plane  in $\E$ is trivial.
\qed
\end{lem}

\subsection{Intersection of  $\ell\leq k$ transversal elliptic hyperplanes}
\label{sec inters of trans}

We say that $\ell\leq k$ elliptic hyperplanes $\tilde H_1,\dots,\tilde H_\ell$ intersect
transversally, if they are defined by equations
\bean
\label{trans int}
\tilde t_1 a_{1j} + \dots + \tilde t_ka_{kj} - \tilde z_j=0, \qquad j=1,\dots,\ell,
\eean
and  the rank of the  $\ell\times k$-matrix $a=(a_{ij})$
equals $\ell$.

By standard theorems, see for example \cite{Vi},
 there are  coordinates $\tilde t_1',\dots, \tilde t_k'$
on $\E$ such that system \Ref{trans int} is equivalent to a system
\bean
\label{trans int better}
 d_j  \tilde t_j' - \tilde z_j'=0, \qquad j=1,\dots,\ell,
\eean
where $\tilde z_j'\in E$, $d_j\in \Z_{>0}$ and $d_j|d_{j+1}$ for $j=1,\dots,\ell-1$. Therefore, the intersection
$X$ of $\ell$ transversal
hyperplanes in $\E$ consists
of $(d_1\dots d_\ell)^2$ parallel elliptic $k-\ell$-planes.

\medskip
Let $\bs w=(w_1,\dots,w_k)$ be complex numbers  and $\LLw$  the  line bundle over $\E$ with
a flat connection defined in Section \ref{sec 2}. We say that $\bs w$ are {\em convenient} for $\E$,
  if there are no nonzero $\LLw$-valued holomorphic differential $k$-forms on $\E$.

\begin{lem}
\label{lem w not in lambda}
Complex numbers $\bs w$ are convenient for $\E$ if and only if $\bs w\notin \Lambda^k$.
\end{lem}

\begin{proof} It is enough to prove the lemma for $k=1$. If $k=1$ and $\omega$ is an
$\LLw$-valued holomorphic 1-form, then it is 1-periodic and has the Fourier series
expansion. The expansion easily implies the required statement.
\end{proof}

Similarly, we say that $\bs w$ are {\em convenient} for the transversal intersection
$X$ with $\dim X=k-\ell>0$,  if  there are no nonzero
 $\LLw$-valued holomorphic differential $k-\ell$-forms on any of the parallel $k-\ell$-planes
 composing $X$.

\begin{lem}
\label{lem w not in lambda X}
Complex numbers $\bs w$ are convenient for the transversal intersection $X$, $\dim X>0$, if and only if
there exist integers $l_1,\dots,l_k$ such that
$\sum_{i=1}l_ia_{ij}=0$ for $j=1,\dots,\ell$, and $l_1w_1+\dots+l_kw_k\notin\Lambda$.
\end{lem}

\begin{proof}
The proof of Lemma \ref{lem w not in lambda X} is the same as the proof of Lemma \ref{lem w not in lambda}.
\end{proof}

\begin{lem}
\label{lem nontr}
If the numbers $\bs w$ are convenient for $\E$, then $H^*(\E,\LLw)=0$. If $\bs w$ are convenient for the transversal intersection
$X$, $\dim X>0$,
 then $H^*(X,\LLw|_X)=0$.
\end{lem}
\begin{proof} The lemma follows from the Kunneth formula and the fact that the cohomology of a circle with
coefficients in a nontrivial local system is zero.
\end{proof}

\subsection{Differential forms of $k$ transversal hyperplanes in $\E$}
\label{sec admis}

This section contains the main construction of the paper.

Let $k$ transversal elliptic hyperplanes $\tilde H_1,\dots, \tilde H_k$ in $\E$ be given by equations
\bean
\label{eqn}
\sum_{i=1}^k \tilde t_ia_{ij}=\tilde z_j\qquad j=1,\dots,k,
\eean
where $\tilde t_1,\dots,\tilde t_k$ are coordinates on $\E$,
$a=(a_{ij})$ is an integer matrix (with nonzero determinant) and
 $\tilde z_1,\dots,\tilde z_k$ are some points of $E$.

 For a complex number $c$, we denote by $\tilde c$ its projection to $E$. In particular, $\tilde 0\in E$ is the projection
 of $0$.
 For given complex numbers $\bs w=(w_1,\dots,w_k)$,  we consider the system of equations
 \bean
\label{dual eqn}
\sum_{j=1}^k a_{ij}\tilde v_j =\tilde w_i \qquad i=1,\dots,k,
\eean
with respect to the unknown $\tilde v_1,\dots,\tilde v_k\in E$. We say that $\bs w$ is
{\em admissible} for $\tilde H_1,\dots, \tilde H_k$ if
 any coordinate $\tilde v_j$ of any solution of \Ref{dual eqn} is not equal to $\tilde 0$.

\begin{lem}
\label{lem adm = conv}
Assume that the numbers $\bs w$ are convenient for each of the transversal intersections
$X_j, j=1,\dots,k$, where $X_j$ is the intersection of the elliptic hyperplanes
$\tilde H_{1},\dots,\tilde H_{j-1}$, $\tilde H_{j+1}, \dots,\tilde H_{k}$, then
$\bs w$ are admissible for $\tilde H_1,\dots, \tilde H_k$.
\end{lem}

\begin{proof}
Let $\tilde v_1,\dots,\tilde v_k$ be any solution of system \Ref{dual eqn}. Let $l_1,\dots,l_k$
be  integers such that
$\sum_{i=1}l_ia_{ij}=0$ for $j=2,\dots,k$, and $l_1w_1+\dots+l_kw_k\notin\Lambda$.
Then $\tilde 0\neq \sum_i l_i\tilde w_i = \sum_{ij} l_ia_{ij}\tilde v_j=
\sum_i l_ia_{i1}\tilde v_1$. Hence $\tilde v_1\neq \tilde 0$. Similarly we prove that $\tilde v_2,\dots,\tilde v_k$
are not equal to $\tilde 0$.
\end{proof}

 Let $\bs w$ be admissible for $\tilde H_1,\dots, \tilde H_k$. Fix complex numbers $z_1,\dots,z_k$ whose projections to $E$
 are $\tilde z_1,\dots,\tilde z_k$.
 For any integers  $A_i,B_i,C_i,D_i$ with $i=1,\dots,k$, we
consider two systems of equations:
\bean
\label{Eqn}
\sum_{i=1}^k u_ia_{ij}= A_j\tau + B_j + z_j, \qquad j=1,\dots,k,
\eean
and
\bean
\label{Dual eqn}
\sum_{j=1}^k a_{ij}v_j= C_i\tau + D_i + w_i, \qquad i=1,\dots,k.
\eean
 The first system is with respect to complex
 numbers $\bs u=(u_1,\dots,u_k)$ and the second system is with respect to  complex numbers
 $\bs v=(v_1,\dots,v_k)$.

To the solution $\bs v=(v_1,\dots,v_k)$ of  \Ref{Dual eqn}, we assign
 the meromorphic  $k$-form on $\C^k$,
 \bean
 \label{diff}
\omega_{\bs v}(\bs t,\tau) &=& \omega_{\bs v}(t_1,\dots,t_k,\tau) =
\\
&=&
\det a\ \,  e^{-2\pi \sqrt{-1}\sum_{i=1}^k C_it_i}
 \prod_{j=1}^k \sigma_{v_j}(\sum_{i=1}^kt_ia_{ij}-z_j,\tau)\, dt_1\wedge\dots\wedge dt_k .
 \notag
\eean
The form is well-defined since the numbers $\bs w$ are admissible for $\tilde H_1,\dots, \tilde H_k$.

\begin{lem}
\label{lem form on LL}
 The form $\omega_{\bs v}(\bs t,\tau)$ descends to an $\LLw$-valued meromorphic form on $\E$,
i.e. $\omega_{\bs v}(\bs t+\gamma,\tau) = \rho_{\bs w}(\gamma) \omega_{\bs v}(\bs t,\tau)$ for $\gamma\in \Gamma^k$.
\qed
\end{lem}

\begin{lem}
\label{lem correctness}
The form $\omega_{\bs v}(\bs t,\tau)$ does not change if  $\bs v$ is changed by an element of $\Lambda^k$.
\qed
\end{lem}

\begin{lem}
\label{lem res}
Let $\bs u=(u_1,\dots,u_k)$ be the solution of system \Ref{Eqn}. Then
\bea
&&
\omega_{\bs v}(t_1+u_1,\dots,t_k+u_k,\tau)  =
\\
&&
\phantom{aaaa}
=  M(\bs u, \bs v)\,
\det a\ \,  e^{-2\pi \sqrt{-1}\sum_{i=1}^k C_it_i}
 \prod_{j=1}^k \sigma_{v_j}(\sum_{i=1}^kt_ia_{ij},\tau)\, dt_1\wedge\dots\wedge dt_k ,
\eea
where
\bea
M(\bs u,\bs v)= e^{2\pi\sqrt{-1}\sum_{i=1}^k(A_iv_i-C_iu_i)}.
\eea
\qed
\end{lem}

For a complex number $c$, we shall write $c=c_\R+\tau c_\tau$ with $c_\R,c_\tau\in\R$.
\begin{lem}
\label{lem M}
We have
\bea
\sum_{i=1}^k(A_iv_i-C_iu_i)
=\sum_{i=1}^k (A_iv_{i,\R}- B_iv_{i,\tau}) + \sum_{i=1}^k (u_i w_{i,\tau}-z_iv_{i,\tau}) .
\eea

\end{lem}

\begin{proof} $\sum_i A_iv_i=\sum_i(A_iv_{i,\R}+\tau A_iv_{i,\tau})$ and
$\sum_i C_iu_i=\sum_i(\sum_j a_{ij}v_{j,\tau} -w_{i,\tau})u_i
=
\linebreak
\sum_{ij}u_{i,\R}a_{ij}v_{j,\tau} + \tau \sum_{ij}u_{i,\tau}a_{ij}v_{j,\tau}
-\sum_i w_{i,\tau}u_i = \sum_j(B_j+z_{j,\R})v_{j,\tau} + \tau\sum_j (A_j+z_{j,\tau})v_{j,\tau}
-\sum_i w_{i,\tau}u_i$. These equalities give the lemma.
\end{proof}

If $\bs u$ is a solution of \Ref{Eqn}, then $p(\bs u)$ is a solution \Ref{eqn}.
All solutions of \Ref{eqn} have this form. Similar relations hold for systems
\Ref{Dual eqn} and \Ref{dual eqn}.

\medskip

Each of the systems \Ref{eqn} and \Ref{dual eqn} has $(\det a)^2$ solutions.
For each solution $\tilde {\bs u}$ of \Ref{eqn} we fix a solution $\bs u$ of \Ref{Eqn} such
that $p(\bs u)=\tilde {\bs u}$. We denote by $\mc U$ the constructed set of  $(\det a)^2$ points
$\bs u\in \C^k$.
For each solution $\tilde {\bs v}$ of \Ref{dual eqn} we fix a solution $\bs v$ of \Ref{Dual eqn} such
that $p(\bs v)=\tilde {\bs v}$. We denote by $\mc V$ the constructed set of  $(\det a)^2$ points
$\bs v\in \C^k$.

\begin{thm}
\label{thm det}
The matrix $M=(M(\bs u, \bs v))_{\bs u\in\mc U, \bs v\in\mc V}$ is nondegenerate.
\end{thm}
\begin{proof}
Let $M_1(\bs u, \bs v)=e^{2\pi\sqrt{-1}\sum_{i=1}^k (A_iv_{i,\R}- B_iv_{i,\tau})}$. The matrix
$M_1=(M_1(\bs u, \bs v))_{\bs u\in\mc U, \bs v\in\mc V}$ is obtained from
$M$ by multiplication by nondegenerate diagonal matrices. Thus, it is enough to prove that $M_1$ is nondegenerate.

The nondegeneracy of $M_1$ follows from the nondegeneracy of $M_1$ for $\bs w=0$ and $\bs z=0$, since the matrix $M_1$ for $\bs w$, $\bs z$
not necessarily equal to zero is obtained
from the matrix $M_1$ with $\bs w=0$, $\bs z = 0$ by multiplication by nondegenerate diagonal matrices.

By elementary row and column transformations, the pair of systems \Ref{Eqn} and  \Ref{Dual eqn} can be reduced to the case
of a diagonal matrix $a$. For a diagonal $a$ and $\bs w=0$, $\bs z=0$ the  nondegeneracy of $M_1$ is obvious.
\end{proof}

\begin{thm}
\label{cor exist forms}
Let $\bs w$ be admissible for $\tilde H_1,\dots, \tilde H_k$. Let $\mc U$ be a set as above.
 Then there exist the  unique differential $k$-forms  $\omega_{\bs u, \tilde H_1,\dots, \tilde H_k}(\bs t,\tau)$,
  $\bs u\in\mc U$, such that
each $\omega_{\bs u, \tilde H_1,\dots, \tilde H_k}(\bs t,\tau)$
is a $\C$-linear combination of forms $\omega_{\bs v}(\bs t,\tau), \bs v\in\mc V$,
and for any $\bs u,\bs u'\in \mc U$ we have the followings expansion,
\bea
\omega_{\bs u, \tilde H_1,\dots, \tilde H_k}(\bs t+\bs u',\tau) = (\delta_{\bs u,\bs u'}
+ \mc O(\bs t))\, d \log (\sum_{i=1}^k t_ia_{i1})\wedge\dots\wedge d \log  (\sum_{i=1}^k t_ia_{ik}) ,
\eea
where $\mc O(\bs t)$ is a function holomorphic at $\bs t=0$ and $\mc O(\bs 0)=0$.

\end{thm}

\begin{proof}
The theorem is a direct corollary of Theorem \ref{thm det}.
\end{proof}

Given transversal $\tilde  H_1,\dots, \tilde H_k$, the set $\mc U$ is not unique, each point ${ \bs u}'\in\mc U$
can be shifted by any element $\gamma=(l_1\tau+m_1,\dots, l_k\tau+m_k)$ of $ \Lambda^k$.

\begin{lem}
\label{lem on u dependence}
Assume that exactly one point  $\bs u'$ of the set $ \mc U$ is replaced with a point $\bs u''=\bs u'+\gamma$.
Consider the set of differential forms assigned to the new set $\mc U$ by Theorem \ref{cor exist forms}.
Then $\omega_{\bs u'',\tilde H_1,\dots, \tilde H_k}(\bs t,\tau) =
e^{2\pi\sqrt{-1}(w_1l_1+\dots+w_kl_k)}\omega_{\bs u', \tilde H_1,\dots, \tilde H_k}(\bs t,\tau)$
and all other differential forms $\omega_{\bs u, \tilde H_1,\dots, \tilde H_k}(\bs t,\tau)$, $\bs u\in \mc U$, remain unchanged.
\qed
\end{lem}

\subsubsection{{\bf The residue of $\omega_{\bs u, \tilde H_1,\dots, \tilde H_k}$}}

Let $H$ be a hyperplane in $\C^k$ defined by an equation
$t_1a_{1j}+\dots+t_ka_{kj}=A_j\tau + B_j+ z_j$, where $j\in \{1,\dots,k\}$ and
$A_j,B_j$ are some integers,
c.f.  \Ref{Eqn} and \Ref{diff}. We have $p(H)=\tilde H_j$.
Let  $\bs u$ be a point of $\mc U$ and
$\omega_{\bs u, \tilde H_1,\dots, \tilde H_k}$ the corresponding differential form. We denote
by $\eta_{\bs u}$ the residue of  $\omega_{\bs u,\tilde H_1,\dots, \tilde H_k}$ at $H$.

\begin{lem}
\label{lem res}
 Assume that a vector
$\gamma = (l_1\tau+m_1,\dots,l_k\tau+m_k)\in\Lambda^k$ is tangent to $H$, i.e.
$\sum_i(l_i\tau+m_i)a_{ij}=0$. Then for all $\bs t\in H$,
 we have
 $\eta_{\bs u}(\bs t +\gamma)=e^{2\pi\sqrt{-1}(w_1l_1+\dots+w_kl_k)}\eta_{\bs u}(\bs t)$.
 That is, the form $\eta_{\bs u}$ defines an $\LLw$-valued differential form
  over the elliptic hyperplane $p(H)=\tilde H_j\subset \E$.
\end{lem}

Now we choose $H$ in Lemma \ref{lem res} so that $\bs u\in H$.

For $i\neq j$, the intersection $\tilde H_i\cap \tilde H_j$ is a collection of parallel elliptic
$k-2$-planes. We denote by  $\tilde H_{i}^{(j)}$ that elliptic $k-2$-plane which contains $\tu=p(\bs u)$.
Then $\tilde H_{1}^{(j)}, \dots, \tilde H_{j-1}^{(j)}, \tilde H_{j+1}^{(j)},\dots,\tilde H_{k}^{(j)}$
are transversal elliptic hyperplanes in $\tilde H_j$.

\begin{thm}
\label{thm funct} We have
$\eta_{\bs u}(\bs t)=(-1)^{j-1}\omega_{\bs u, \tilde H_{1}^{(j)}, \dots, \tilde H_{j-1}^{(j)}, \tilde H_{j+1}^{(j)},\dots,\tilde H_{k}^{(j)}}$ .
\end{thm}

\begin{proof}
The difference of the right hand side and the left hand side defines an $\LLw$-valued form on $\tilde H_j$
with logarithmic singularities along
$\tilde H_{1}\cap \tilde H_j, \dots, \tilde H_{j-1}\cap \tilde H_j, \tilde H_{j+1}\cap \tilde H_j,\dots, \tilde H_{k}\cap \tilde H_j$.
The difference has zero $k-1$-iterated residues at all points.
Therefore, the difference vanishes due to the following lemma.
\end{proof}

\begin{lem}
\label{lem induction}
Assume that for every $i=1,\dots,k$, we have a finite set of parallel hyperplanes $\{\tilde H_i^{l_i}\ |\ l_i\in L_i\}$ in $\E$.
Assume that the hyperplanes $\tilde H^{l_1}_1,\dots,\tilde H_k^{l_k}$ intersect  transversally.
Assume that numbers $\bs w$ are convenient for the transversal intersection of
$\tilde H^{l_1}_1,\dots,\tilde H_k^{l_k}$.

Let $\Omega$ be an $\LLw$-valued meromorphic differential $k$-form on $\E$ with logarithmic singularities at the union of
all hyperplanes   $\{\tilde H_i^{l_i}\ |\ i=1,\dots,k,\,l_i\in L_i\}$. Assume that $\Omega$ has zero
$k$-iterated residues at all points of $\E$. Then $\Omega$ is the zero form.
\end{lem}
\begin{proof} The proof is by induction on $k$. If $k=1$, then $\Omega$ is regular on $E$. Since $\bs w$ are convenient,
$\Omega$ vanishes.

Step of the induction. The residue of $\Omega$ at any hyperplane $\tilde H_j^{l_j}$ has the same properties as $\Omega$:
the residue has  logarithmic singularities at the union of all intersections
     $\tilde H_i^{l_i}\cap \tilde H_j^{l_j}$, the residue has  zero $k-2$-iterated residue at any point.
By the induction assumption, the residue of $\Omega$ at $\tilde H_j^{l_j}$ vanishes, hence,
      $\Omega$ is regular  on $\E$ and $\Omega$ is the zero form due to the convenience of $\bs w$.
\end{proof}

\subsubsection{{\bf Example}}

Here is an example illustrating Theorem \ref{cor exist forms} for $k=1$. Consider an
analog of the pair of systems \Ref{eqn} and \Ref{dual eqn}: $2\tilde t=0$ and $2\tilde v = \tilde w$, where $w\notin \Lambda$.
We can choose $\mc U = \{ 0, 1/2,\tau/2, 1/2+\tau/2\}$ and
$\mc V = \{ w/2,  w/2+1/2, w/2+\tau/2,  w/2+1/2+\tau/2\}$.
The differential forms $\omega_{\bs v}$, $\bs v\in\mc V$, given by formula \Ref{diff}, are
\bea
\omega_1=2 \sigma_{w/2}(2t,\tau)dt,
&\qquad &
\omega_2=2 \sigma_{w/2+1/2}(2t,\tau)dt,
\\
\omega_3=2 e^{-2\pi\sqrt{-1}t} \sigma_{w/2 + \tau/2}(2t,\tau)dt,
&\qquad &
\omega_4=2 e^{-2\pi\sqrt{-1}t} \sigma_{w/2+1/2+\tau/2}(2t,\tau)dt.
\eea
Denote $\gamma=e^{-\pi\sqrt{-1}w}$.
Then the differential forms $\omega_{\bs u}$, $\bs u\in\mc U$, given by Theorem \ref{cor exist forms}, are
\bea
\tilde \omega_1= \frac 14 (\omega_1+  \omega_2+
\gamma  \omega_3+  \gamma  \omega_4) ,
\qquad
\tilde \omega_2= \frac 14 (\omega_1+  \omega_2
-\gamma  \omega_3-\gamma  \omega_4) ,
\eea
\bea
\tilde \omega_3= \frac 14 (\omega_1- \omega_2
+\gamma  \omega_3- \gamma  \omega_4) ,
  \qquad
  \tilde \omega_4= \frac 14 (\omega_1-  \omega_2
- \gamma  \omega_3+\gamma \omega_4) .
\eea
The forms $\tilde \omega_i$, $i=1,\dots,4$, define meromorphic $\LLw$-valued
differential forms on $E$.
The form $\tilde \omega_1$  is regular on $\C-\Lambda$, has simple poles at $\Lambda$,
has residue 1 at $t=0$. The forms $\tilde \omega_2, \tilde \omega_3, \tilde \omega_4$
 have similar properties with respect to the
sets  $\C-(1/2+\Lambda)$, $\C-(\tau/2+\Lambda)$, $\C-(\tau/2+1/2+\Lambda)$ and points
$1/2,\tau/2,\tau/2+1/2$, respectively.
These properties imply that
\bea
\tilde \omega_1= \sigma_{w}(t,\tau)dt,
&\qquad &
\tilde \omega_2= \sigma_{w}(t-1/2,\tau)dt,
\\
\tilde \omega_3= \sigma_{w}(t-\tau/2,\tau)dt,
&\qquad &
\tilde\omega_4 = \sigma_w(t-\tau/2-1/2,\tau)dt.
\eea

\section{Arbitrary elliptic arrangement}
\label{proofs}

\subsection{An elliptic arrangement}

An elliptic arrangement in $\E=E^k$ is  a finite collection $\CC = \{\tilde H_j\}_{j\in J}$ of elliptic hyperplanes.
We fix coordinates $\tilde t_1,\dots,\tilde t_k$ on $\E$ and for every $j\in J$  we fix  an equation
$\tilde t_1a_{1j} + \dots + \tilde t_ka_{kj} - \tilde z_j=0$ defining the hyperplane $\tilde H_j$.

We denote by
$$
\tilde M_\CC = \E - \cup_{j\in J} \tilde H_j \ ,
$$
the complement of the arrangement.

Consider  the intersection of any $\ell\leq k$ transversal hyperplanes of $\CC$. The intersection consists
of a finite set of parallel elliptic $k-\ell$ planes. Each of these $k-\ell$-planes will be called
an {\rm edge} of $\E$.
 In particular, if $\ell=k$, then the $0$-planes will be called {\em vertices}
of $\E$.

 For an edge $X$ we denote
$J_{X} =\{j\in \subset J\ |\ X\subset \tilde H_j\}$.

We denote by $\tilde{\mc U}$ the set of all vertices of $\CC$.
For every vertex
$\tu\in \tilde{\mc U}$ we choose a point $\bs u\in\C^k$ such that $p(\bs u) = \tu$.
The set of all  chosen points in $\C^k$ is denoted by $\mc U$.

\medskip
We say that complex numbers $\bs w=(w_1,\dots,w_k)$ are {\em convenient} for the elliptic arrangement $\CC$,
if $\bs w$ are convenient for the intersection of every $\ell<k$ transversal  hyperplanes $\tilde H_{j_1}, \dots,
\tilde H_{j_\ell}$ of $\CC$ (in the sense of Section \ref{sec inters of trans}).

\subsection{Differential $k$-forms of an elliptic arrangement}
\label{sec diff forms}

For a vertex $\tilde{\bs u}\in \tilde{\mc U}$, we denote by $\CC_{\tilde{\bs u}}=\{\tilde H_j\}_{j\in I_{\tilde{\bs u}}}$
the subarrangement of all hyperplanes of
$\CC$ containing $\tilde{\bs u}$.
In a small neighborhood of $\tilde{\bs u}$ the arrangement $\CCu$ is isomorpic to a central arrangement of affine hyperplanes.
 We denote by $\OS_\tu^k$ the $k$-th graded component of the Orlik-Solomon algebra of that arrangement.
More precisely, let $\OS_\tu^k$ be the complex vector space generated by symbols
$(\tilde H_{j_1},..., \tilde H_{j_k})$ with ${j_i}\in J_\tu$, subject to the relations:
\begin{enumerate}
\item[(i)] $(\tilde H_{j_1},...,\tilde H_{j_k})=0$
if $\tilde H_{j_1}$,...,$\tilde H_{j_k}$ are not transversal;
\item[(ii)]
$ (\tilde H_{j_{\sigma(1)}},..., \tilde H_{j_{\sigma(k)}})=(-1)^{|\sigma|}
(\tilde H_{j_1},..., \tilde H_{j_k})
$
for any  $\sigma\in S_k$;
\item[(iii)]
$\sum_{i=1}^{k+1}(-1)^i (\tilde H_{j_1},...,\widehat{\tilde H_{j_i}},...,\tilde H_{j_{k+1}}) = 0$ for
any $k+1$ elliptic hyperplanes of $\CCu$.

\end{enumerate}
We set
\bea
\OS_\CC^k = \oplus_{\tu\in \tilde{\mc U}} \OS_\tu^k .
\eea

\medskip

Let us fix  $\bs w=(w_1,\dots,w_k)$ convenient for $\E$. Let $\tilde{\bs u}\in \tilde{\mc U}$ and  ${\bs u}\in {\mc U}$ be
such that $p(\bs u)= \tilde{\bs u}$. Let $\tilde H_{j_l},\dots,\tilde H_{j_k}$ be any $k$ transversal hyperplanes in $\CCu$.
Denote by $\omega_{\bs u, \tilde H_{j_l},\dots, \tilde H_{j_k}}(\bs t,\tau)$ the differentail meromorphic $k$-form on $\C^k$
assigned by Theorem \ref{cor exist forms} to these $k$ transversal hyperplanes and denoted by
$\omega_{\bs u, \tilde H_1,\dots, \tilde H_k}(\bs t,\tau)$ in  Theorem \ref{cor exist forms}.
We denote by $A_{\tilde{\bs u}}^k$ the complex vector space generated by the forms
 $\omega_{\bs u; \tilde H_{j_l},\dots, \tilde H_{j_k}}(\bs t,\tau)$.
 Notice that by Lemma \ref{lem on u dependence}, the space $A^k_{\tilde{\bs u}}$ does not depend on the choice of $\bs u$ such
 $p(\bs u)=\tilde{\bs u}$.

 We denote by $A_\CC^k$ the sum of vector spaces
 $A_{\tilde{\bs u}}^k$, $\tilde{\bs u} \in \tilde{\mc U}$.

 \begin{thm}
 \label{thm ell OS}
 ${}$

 \begin{enumerate}

\item[(i)]

  The map $\OS_{\tilde{\bs u}}^k\to A^k_{\tilde{\bs u}},\  (\tilde H_{j_l},\dots, \tilde H_{j_k})
  \mapsto \omega_{\bs u; \tilde H_{j_l},\dots, \tilde H_{j_k}}(\bs t,\tau),$
   is an isomorphism of vector spaces.

\item[(ii)]
We have $A^k_\CC = \oplus_{\tilde{\bs u}\in \tilde{\mc U}}A^k_{\tilde{\bs u}}$.

 \end{enumerate}

 \end{thm}
\begin{proof}
It is enough  to prove that for any  $k+1$ elliptic hyperplanes of $\CCu$, we have the elliptic Orlik-Solomon
relation
\bean
\label{ell OS}
\sum_{i=1}^{k+1}(-1)^i \omega_{\bs u,\tilde H_{j_1},...,\widehat{\tilde H_{j_i}},...,\tilde H_{j_{k+1}}} = 0 .
\eean
The proof is by induction on $k$.
If $k=1$, the difference $\omega_{\bs u, \tilde H_1}-\omega_{\bs u, \tilde H_2}$ is regular on $E$ and  is the zero
1-form due to the convenience of $\bs w$.

Step of the induction. For every $i=1,\dots,k+1$, the residue at $\tilde H_{j_i}$ of the left hand side in \Ref{ell OS} is
the left hand side of an elliptic Orlik-Solomon relation for an arrangement in $\tilde H_{j_i}$, see Theorem \ref{thm funct}.
By the induction assumption, the residue of the left hand side at $\tilde H_{j_i}$ is the zero $k-1$-form. Hence, the left hand side in
\Ref{ell OS} is regular on $\E$ and vanishes due to the convenience of $\bs w$.
\end{proof}

\subsection{Cohomology of the complement}
\label{Cohomology of the complement} Every form $\omega\in A^k_\CC$ induces a holomorphic
\linebreak
$\LLw$-valued $k$-form
$p_*(\omega)$ on the complement $\tilde M_\CC$ of the elliptic arrangement $\CC$. The image of $A^k_\CC$ will be denoted
by $p_*(A_\CC^k)$.  The assignment to $p_*(\omega)$ its cohomology class
 $[p_*(\omega)]$ defines a linear map $\iota : A^k_\CC \to H^k(\tilde M_\CC; \LL_w)$.
 Here $H^*(\tilde M_\CC;\LLw)$ denotes the cohomology of $\tilde M_\CC$ with coefficients in
the local system of horizontal sections of $\LLw$.

\begin{thm}
\label{Thm main}
Assume that $\bs w$ are convenient for $\CC$. Then $H^\ell(\tilde M_\CC;\LLw)=0$ for $\ell\neq k$ and
$\iota_{\CC} : A_\CC^k \to H^k(\tilde M_\CC; \LLw)$ is an isomorphism.
\end{thm}

\begin{proof} We need the following lemmas.

\begin{lem}
\label{lem iso}
The map $\iota_{\CC} : A_\CC^k \to H^k(\tilde M_\CC; \LLw)$ is a monomorphism.
\end{lem}
\begin{proof}
For a central affine arrangement of hyperplanes in $\C^k$, the $k$-th homology group
of the complement with trivial coefficients is generated by $k$-dimensional tori located
near the vertex of the arrangement and corresponding to the $k$-flags of the arrangement, see
Section 4.4 in \cite {SV}. The nondegenerate pairing between the top degree
cohomology of the complement
with trivial coefficients and the top degree homology is given by
the integrals of the Orlik-Solomon differential forms over the tori. The integrals are nothing else but
the multiple residues of the differential forms at the flags of the arrangement.
Locally at $\tu\in \tilde{\mc U}$, the arrangement $\CC_\tu$ is isomorphic to a central affine arrangement.
The $k$-dimensional tori of that central arrangement, considered as $k$-dimensional tori in a small
neighborhood of $\tu$ in $\tilde M$ induces a vector subspace $H_{\tu,k}\subset H_k(\tilde M;\LLw)$. The pairing between
$H_{\tu,k}$ and $A_{\tilde {\bs u}'}$ is zero if $\tu\neq \tilde {\bs u}'$ and the pairing is nondegenerate if
$\tu=\tilde {\bs u}'$.
\end{proof}

\begin{lem}
\label{lem k=1}
Theorem \ref{Thm main} is true for $k=1$.
\end{lem}

\begin{proof}
The lemma follows from the convenience of $\bs w$
and the exact sequence for the pair $\tilde M_\CC \subset E$.
\end{proof}

Let $j_0$ be an element of $J$. We consider the following three elliptic arrangements:
 $\CC, \CC', \CC''$, where
$\CC'=\{\tilde H_j\}_{j\in J-\{j_0\}}$ and $\CC''$ is the elliptic arrangement induced by $\CC$ on
$\tilde H_{j_0}$.

\begin{lem}
\label{OS exact}
We have an exact sequence
\bean
\label{ex sec}
0\to A_{\CC'}^k\to A_\CC^k\to A_{\CC''}^{k-1}\to 0 ,
\eean
where the second map is the residue at $\tilde H_{j_0}$.
\end{lem}

\begin{proof}

The lemma follows from the fact that $A_\CC^k, A_{\CC'}^k, A_{\CC''}^{k-1}$ are isomorphic to the top degree components of the Orlik-Solomon
algebras of central arrangements.
\end{proof}

\begin{lem}
\label{ lem main}
${}$

\begin{enumerate}
\item[(i)]

For $\ell\neq k$ we have $H^\ell(\tilde M_{\CC'}; \LLw)=H^\ell(\tilde M_{\CC}; \LLw)=
       H^{\ell-1}(\tilde M_{\CC''}; \LLw)=0$.
\item[(ii)]
Consider the following diagram
       \begin{equation*}
\begin{gathered}\xymatrix@C=30pt{
    0
    \ar[r]
    &
    A^k_{\CC'}
    \ar[r]
    \ar[d]^{\iota_{\CC'}}
        &
    A^k_{\CC}
   \ar[r]
    \ar[d]^{\iota_{\CC}}
    &
        A^{k-1}_{\CC''}
        \ar[r]
     \ar[d]^{\iota_{\CC''}}
    &
    0
    \\
0
    \ar[r]
    &
       H^k(\tilde M_{\CC'}; \LLw)
    \ar[r]
    &
       H^k(\tilde M_{\CC}; \LLw)
    \ar[r]
    &
       H^{k-1}(\tilde M_{\CC''}; \LLw)
    \ar[r]
    &
    0
}\end{gathered}
\end{equation*}
where  the top horizontal sequence is the sequence \Ref{ex sec}, the homomorphisms of the bottom horizontal sequence
are  the homomorphisms of the exact sequence of the pair $\tilde M_{\CC} \subset \tilde M_{\CC'}$,
c.f. Lemma \ref{lem normal}. Then the diagram is commutative,
the horizontal sequence is exact and the vertical homomorphisms are isomorphisms.
\end{enumerate}
\end{lem}

\begin{proof} The proof of this lemma is similar to the corresponding proofs in Section 5.4 of \cite{OT}.
Namely, using Lemma \ref{lem normal}, one proves that there is a cohomology long exact sequence
\bea
  \dots \to   H^\ell(\tilde M_{\CC'}; \LLw) \to
       H^\ell(\tilde M_{\CC}; \LLw)\to
       H^{\ell-1}(\tilde M_{\CC''}; \LLw)\to
       H^{\ell+1}(\tilde M_{\CC'}; \LLw)
 \to
       \dots ,
\eea
c.f. Corollary 5.81 in \cite{OT}. Using the induction on $k$, one proves that
$ H^\ell(\tilde M_{\CC'}; \LLw) \simeq       H^\ell(\tilde M_{\CC}; \LLw)$ if $\ell\neq k$.
Using the induction on the number of hyperplanes in $\CC$, one concludes that
$H^\ell(\tilde M_{\CC}; \LLw)=0$ if $\ell\neq k$ and one gets an exact sequence
$  0 \to   H^k(\tilde M_{\CC'}; \LLw) \to
       H^k(\tilde M_{\CC}; \LLw)\to
       H^{k-1}(\tilde M_{\CC''}; \LLw)\to
       0$. Then using the double induction on $k$ and the number of hyperplanes in $\CC$ one gets
       the second statement of Lemma \ref{ lem main}.
\end{proof}

Lemma \ref{ lem main} implies Theorem \ref{Thm main}.
\end{proof}

Theorem \ref{Thm main} implies Theorems \ref{thm on forms} and \ref{thm main}.

\end{document}